\documentclass[12pt]{article}

\usepackage{setspace}
\usepackage{amsthm, amssymb, amstext}
\usepackage[fleqn]{amsmath}
\usepackage{latexsym}
\usepackage{comment}
\usepackage{hyperref}
\usepackage{mathtools}
\usepackage{enumerate} 
\usepackage{times}

\usepackage{todonotes}
\usepackage{comment}

\newtheorem{theorem}{Theorem}
\newtheorem{lemma}{Lemma}

\newtheorem{observation}{Observation}
\newtheorem{corollary}{Corollary}

\theoremstyle{definition}

\usepackage{amsthm}

\title{Solution to a problem of Gr\"unbaum on the edge density of $4$-critical planar graphs}
\author{Zden\v{e}k Dvo\v{r}\'ak\thanks{Computer Science Institute (CSI) of Charles University, Prague, Czech Republic, \url{rakdver@iuuk.mff.cuni.cz}.}
\and Carl Feghali\thanks{Univ Lyon, EnsL, CNRS, LIP, F-69342, Lyon Cedex 07, France, email: \url{carl.feghali@ens-lyon.fr}.}}
\date{}

\begin{document}
\maketitle

\begin{abstract}
We show that $\limsup |E(G)|/|V(G)| = 2.5$ over all $4$-critical planar graphs $G$, answering a question of Gr\"unbaum from 1988.
 \end{abstract}

 \section{Introduction}

For a positive integer $k$, a \emph{(proper) $k$-coloring} of a graph $G = (V, E)$ is a function $f: V \rightarrow \{1, \dots, k\}$ such that $f(u) \not= f(v)$ whenever $uv \in E$. A graph $G$ is \emph{$k$-colorable} if it has a $k$-coloring. The \emph{chromatic number}, $\chi(G)$, of a graph $G$ is the smallest $k$ such that $G$ is $k$-colorable. A  graph $G = (V, E)$ is \emph{$k$-critical} if $\chi(G - e) < \chi(G) = k$ for every  edge $e \in E$. 

Recognizing $k$-colorable graphs is a well-known NP-complete problem for every $k \geq 3$. Thus, little can probably be said about their structure unless $\mbox{P} = \mbox{NP}$. One obvious structure, however, that a $k$-chromatic graph must have is that it contains a $k$-critical subgraph. For a $k$-critical graph, more can be said: the minimum and maximum number $f(k, n)$ and $F(k, n)$, respectively, of edges over all $k$-critical graphs on $n$ vertices must be bounded.  It is then perhaps not surprising that the problem of determining $f(k, n)$ or $F(k, n)$  has a very rich history. 

In particular, although we are concerned here  with the function $F(k, n)$, we remark that finding $f(k, n)$ was first asked by Dirac in 1957 \cite{dirac} and reiterated by Gallai \cite{gallai1963} in 1963, Ore~\cite{ore} in 1967 and others. It also remains an active area of research; for instance, in a recent breakthrough, Kostochka and Yancey \cite{kostochka} solved a 1963 conjecture of Dirac that $f(k, n) \geq \lceil \frac{(k+1)(k-2)n - k(k-3)}{2(k-1)} \rceil$.

On the other hand, the function $F(k, n)$ turns out to be less well-behaved. Toft \cite{toft} proved, in particular, that for every $k \geq 4$ there are $k$-critical graphs $G$ on $n$ vertices with $|E(G)| \geq c_k n^2$ for some constant $c_k >0$ (despite this, $k$-critical graphs do have many other interesting properties; see \cite{stiebitz}). 

But what if we further restrict our focus to special classes of graphs?  Of relevance, determining the maximum number,  $H(4, n)$,  of edges over all $4$-critical planar graphs has received considerable attention.
A motivation for the study of this parameter comes from
a possible pre-processing step in (superpolynomial-time) algorithms
for 3-coloring planar graphs: If the input $n$-vertex planar graph $G$ has
more than $H(4, n)$ edges, then there necessarily exists an edge $e$ such 
that $\chi(G)=\chi(G-e)$; moreover, the known upper bounds on $H(4,n)$ can
be easily turned into polynomial-time algorithms to find such an edge.
Hence, this gives us a way to reduce the size of the input (the number
of edges) before applying more time-consuming algorithms.

Let us note that Euler's formula gives a trivial upper bound $H(4, n) \leq 3n - 6$.  In 1964, Gallai~\cite{gallai1963c} conjectured that $H(4, n) \leq 2n - 2$. In 1985, Koester \cite{koester1} disproved this conjecture, constructing a $4$-critical $4$-regular planar graph. Motivated by this line of work, Gr\"unbaum \cite{grunbaum} defined and asked to determine the precise value of the function
$$L\coloneqq\limsup |E(G)|/|V(G)|,$$
over all $4$-critical planar graphs $G$ and, in that same paper, showed that $L \geq 79 / 39$ (improving the aforementioned bound of $2$ by Koester). This was, in turn, soon improved, though marginally, by Koester \cite{koester2}. Recently Yao and Zhou~\cite{yao} established that $L \geq 7 / 3$, and conjectured that this is the right answer.

In the other direction, in 1991, Abbott and Zhou \cite{abbott} showed that $L \leq 2.75$. And, in that same year, this was improved by Koester \cite{koester2} to $L \leq 2.5$. The object of this note is to show that Koester's bound is best possible. 

\begin{theorem}\label{thm:main}
    $L = 2.5$. 
\end{theorem}

We prove Theorem \ref{thm:main} by constructing, for each integer $m$
sufficiently large, a $4$-critical planar graph $P_m$ on at most $m$ vertices
such that $$\lim_{m \rightarrow \infty} \frac{|E(P_m)|}{m} = 2.5.$$  This
clearly implies Theorem \ref{thm:main} and thus answers Gr\"unbaum's
aforementioned question. It also refutes the aforementioned conjecture of Yao and
Zhou~\cite{yao} asserting that $L = 7/3$.   

Li et al.~\cite{li20223} studied the dual problem of density of 3-flow-critical graphs. 
The dual to a 4-critical plane graph is 3-flow-critical.  By Theorem~\ref{thm:main},
there exist 4-critical plane graphs with $n$ vertices and $2.5n+o(n)$ edges.
By Euler's formula, such graphs have $1.5n+o(n)$ faces.  Hence, the dual is
a 3-flow-critical graph with $m=1.5n+o(n)$ vertices and $2.5n+o(n)=\tfrac{5}{3}m+o(m)$
edges.  This gives the current best construction of sparse 3-flow-critical graphs
(improved from $\tfrac{7}{4}m+o(m)$ which follows by dualizing the construction of Yao and
Zhou~\cite{yao}), and this is the best possible for planar 3-flow-critical graphs.
Let us remark that Li et al.~\cite{li20223} proved that every (not necessarily planar)
3-flow-critical graph with $m$ vertices has edge density at least $1.6+o(m)$, and thus
$$1.6\le \liminf_{\text{$G$ 3-flow-critical}} \frac{|E(G)|}{|V(G)|}\le \tfrac{5}{3}.$$

\section{The gadgets}

Let us start by describing properties of several graphs that
are combined in our construction.

For a positive integer $n$, let $L_n$ be the graph depicted in Figure \ref{fig:lattice}. Its vertices are labelled alternately $x_{1, 1}, x_{1,2}, \dots, x_{1, n}$ then $y_{1, 1}, y_{1,2}, \dots, y_{1, n+1}$ followed by $x_{2, 1}, x_{2, 2}, \dots, x_{2, n}$ etc. up until $y_{n-1, 1}, y_{n-1, 2}, \dots, y_{n-1, n+1}$ then $x_{n, 1}, x_{n, 2}, \dots, x_{n, n}$ in a left--right and top--bottom fashion as partially illustrated in Figure \ref{fig:lattice}.
For $i\in[n-1]$, we say that $\{y_{i,1},\ldots,y_{i,n+1}\}$ is the $i$-th \emph{stripe} of $L_n$.  In the first lemma, we record some properties of the graph $L_n$. 

\begin{figure}
\begin{center}\includegraphics[width=1\textwidth]{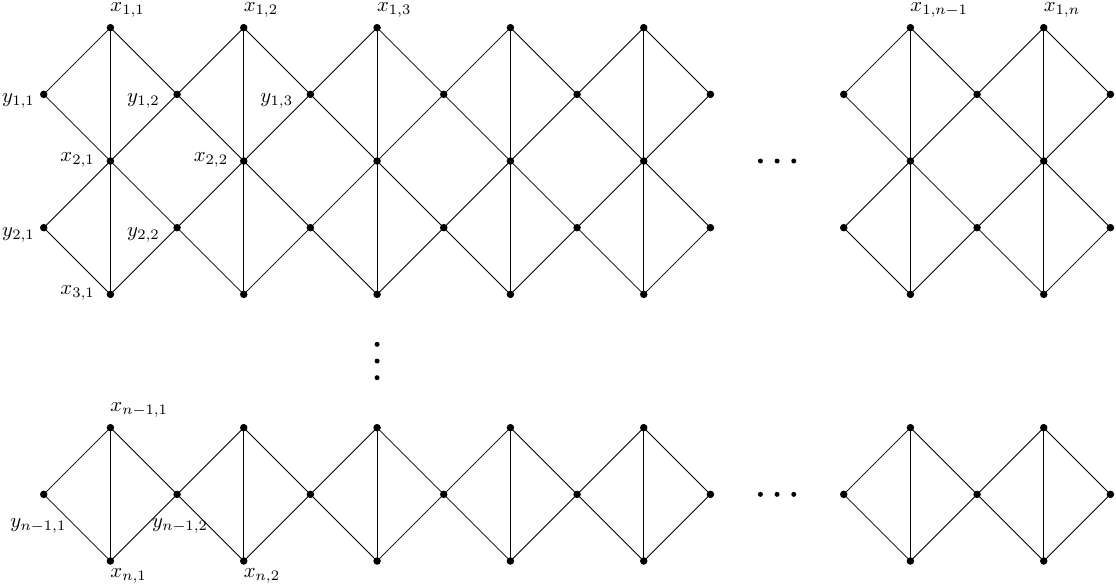}\end{center}
\caption{The graph $L_n$}
\label{fig:lattice}
\end{figure}

\begin{lemma}\label{lem:L}
Let $n\ge 2$ be an integer. For every $3$-coloring $\varphi$ of $L_n$ and every $j_1,j_2 \in [n]$,
$\varphi(x_{1, j_1}) = \varphi(x_{1, j_2})$ if and only if  $\varphi(x_{n, j_1}) = \varphi(x_{n, j_2})$. 
\end{lemma}
\begin{proof}
Note that for $i\in[n-1]$, there exists $c_i\in[3]$ such that all vertices of the $i$-th stripe have color $c_i$.
Moreover, observe that for $i\in[n-1]$, $\varphi(x_{i+1,j_1})$ is the unique
color different from $c_i$ and $\varphi(x_{i,j_1})$, and consequently,
$\varphi(x_{i+1, j_1}) = \varphi(x_{i+1, j_2})$ if and only if $\varphi(x_{i, j_1}) = \varphi(x_{i, j_2})$.  This clearly implies the
claim of the lemma.
\end{proof}

In the following lemmas, we record some properties of the graphs $G_0$, $G_1$, and $G_2$ that are depicted in Figures~\ref{fig:gadget0}, \ref{fig:gadget1prime},
\ref{fig:gadget1} and \ref{fig:gadget2}, respectively. We are going to use the following observation repeatedly.
\begin{observation}\label{obs:star}
Let $F$ be a graph consisting of a triangle $xyz$ and vertices $u$ and $v$ adjacent to $y$ and $z$, respectively.
In any $3$-coloring $\varphi$ of this graph, if $\varphi(u) = \varphi(v)$, then $\varphi(x) = \varphi(u) (=\varphi(v))$.
\end{observation}

Let us now analyze the gadget depicted in Figure~\ref{fig:gadget0}.
\begin{lemma}\label{lem:g0}
Let $G_0$ be the graph depicted in Figure~\ref{fig:gadget0}.  A 3-coloring $\varphi$ of $\{v_1,v_2,v_3,v_4\}$
extends to a 3-coloring of $G$ if and only if either
\begin{itemize}
\item $\varphi(v_1)=\varphi(v_2)=\varphi(v_3)=\varphi(v_4)$, or
\item $\varphi(v_1)\neq \varphi(v_2)$ and $\varphi(v_3)\neq\varphi(v_4)$.
\end{itemize}
\end{lemma}
\begin{proof}
Let us first consider the fragment $F_0$ of $G_0$ depicted on the right in Figure~\ref{fig:gadget0}.
We claim that a $3$-coloring $\psi$ of $\{v_1,v_2,w\}$ extends to a $3$-coloring of $F_0$ if and only if either
\begin{itemize}
\item $\psi(v_1)=\psi(v_2)=\psi(w)$, or
\item $\psi(v_2)\neq \psi(v_1)\neq \psi(w)$.
\end{itemize}
Indeed, the three $3$-colorings of $F_0$ corresponding (up to renaming of the colors) to these possibilities are
shown in Figure~\ref{fig:gadget0cols}.  On the other hand, consider any $3$-coloring $\psi$ of $F_0$.
If $\psi(v_1)=\psi(v_2)$, then Observation~\ref{obs:star} applied to a subgraph of $F_0$ implies that $\psi(w)=\psi(v_1)$.
Conversely, if $\psi(w)=\psi(v_1)$, then Observation~\ref{obs:star} applied to another subgraph of $F_0$ implies that $\psi(v_1)=\psi(v_2)$.
Hence, the $3$-colorings of $F_0$ satisfy the described property.  We are going to refer to this property as ($\star$).

Note that $G_0$ is obtained by gluing four copies of $F_0$.  If $\varphi(v_1)=\varphi(v_2)=\varphi(v_3)=\varphi(v_4)$,
then we color $w_1$ and $w_2$ by the same color and extend the coloring to the copies of $F_0$ by ($\star$).
If $\varphi(v_1)\neq \varphi(v_2)$ and $\varphi(v_3)\neq\varphi(v_4)$, then we
give $w_2$ a color different from $\varphi(v_1)$, $w_2$ a color different
from $\varphi(w_2)$ and $\varphi(v_4)$, and extend the coloring to the copies of $F_0$ by ($\star$).
Hence, if $\varphi$ satisfies the conditions described in the statement of the lemma, then it extends to
a 3-coloring of $G_0$.

Conversely, consider any $3$-coloring $\varphi$ of $G_0$.  If $\varphi(v_1)=\varphi(v_2)$, then ($\star$) applied
to the four copies of $F_0$ implies that $w_2$, $w_1$, $v_4$, and $v_3$ all receive the color $\varphi(v_1)$.
If $\varphi(v_1)\neq\varphi(v_2)$, then ($\star$) applied to the four copies of $F_0$ implies that
$\varphi(w_2)\neq \varphi(v_1)$, $\varphi(w_1)\neq\varphi(w_2)$, $\varphi(v_4)\neq\varphi(w_1)$,
and finally $\varphi(v_3)\neq \varphi(v_4)$, as required.
\end{proof}

\begin{figure}
\begin{center}\includegraphics[width=0.8\textwidth]{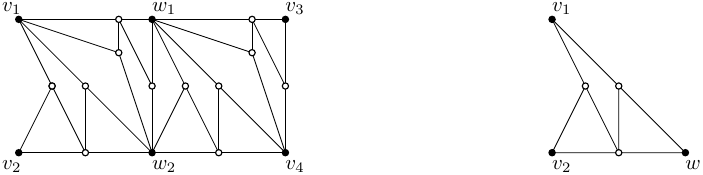}\end{center}
\caption{The graph $G_0$ (left) and its constituent piece $F_0$ (right).}
\label{fig:gadget0}
\end{figure}

\begin{figure}
\begin{center}\includegraphics[width=0.8\textwidth]{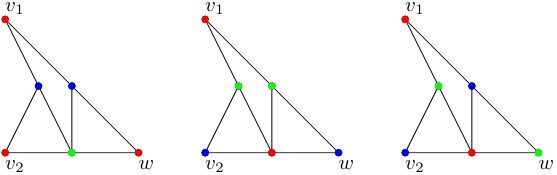}\end{center}
\caption{Useful $3$-colorings of $F_0$.}
\label{fig:gadget0cols}
\end{figure}

\begin{figure}
\begin{center}\includegraphics[width=0.4\textwidth]{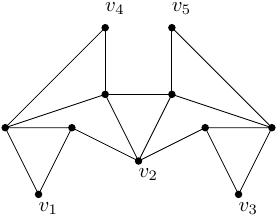}\end{center}
\caption{The graph $G'_1$.}
\label{fig:gadget1prime}
\end{figure}

\begin{figure}
\begin{center}\includegraphics[width=0.8\textwidth]{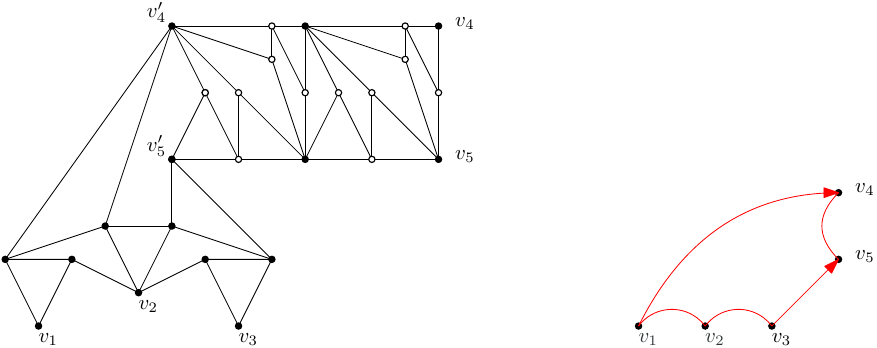}\end{center}
\caption{The graph $G_1$ and its symbolic representation.}
\label{fig:gadget1}
\end{figure}

\begin{figure}
\begin{center}\includegraphics[width=0.8\textwidth]{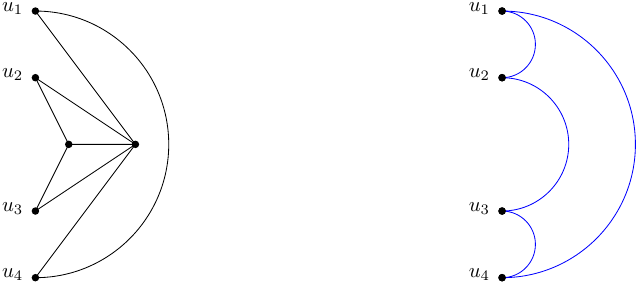}\end{center}
\caption{The graph $G_2$ and its symbolic representation.}
\label{fig:gadget2}
\end{figure}

Next, let us analyze the graph $G_1$ depicted in Figure~\ref{fig:gadget1}, starting from its constituent piece
shown in Figure~\ref{fig:gadget1prime}.

 \begin{lemma}\label{lem:g1prime}
 In every $3$-coloring $\varphi$ of the graph $G'_1$ depicted in Figure \ref{fig:gadget1prime},
 \begin{itemize}
     \item if $\varphi(v_1) = \varphi(v_2) = \varphi(v_3)$, then $\varphi(v_4) = \varphi(v_5) = \varphi(v_2)$; and,
     \item if at least two different colors appear on $v_1$, $v_2$, and $v_3$, then $\varphi(v_4)\neq\varphi(v_5)$.
 \end{itemize}
 Moreover, every $3$-coloring of $\{v_1,v_2,v_3\}$ extends to a $3$-coloring of $G'_1$.
 \end{lemma}
 
 \begin{proof}
  By Observation~\ref{obs:star}, if $\varphi(v_1)=\varphi(v_2)$,
  then $\varphi(v_4)=\varphi(v_2)$, and if $\varphi(v_2)=\varphi(v_3)$, then $\varphi(v_5)=\varphi(v_2)$.  This implies the first claim of the lemma.

  For the second claim of the lemma, let us prove its converse: If $\varphi(v_4)=\varphi(v_5)$, then $\varphi(v_1) = \varphi(v_2) = \varphi(v_3)$.
  Indeed, suppose that $\varphi(v_4)=\varphi(v_5)=1$.  By Observation~\ref{obs:star}, we first conclude that $\varphi(v_2)=1$.
  Applying Observation~\ref{obs:star} again, $\varphi(v_2)=\varphi(v_4)=1$ implies $\varphi(v_1)=1$, and $\varphi(v_2)=\varphi(v_5)=1$ implies $\varphi(v_3)=1$.

  Finally, consider any $3$-coloring $\psi$ of $\{v_1,v_2,v_3\}$.  If
  $\psi(v_1)=\psi(v_2)$, then we extend the coloring to $G'_1$ greedily,
  coloring $v_4$, $v_5$, and the common neighbor of $v_1$ and $v_2$ last.
  Similarly, we can extend the coloring if $\psi(v_2)=\psi(v_3)$.
  If $\psi(v_1)\neq\psi(v_2)\neq \psi(v_3)$, then 
  we color the common neighbors of $v_1$ and $v_4$ and of $v_3$ and $v_5$
  by $\psi(v_2)$, then extend the coloring greedily.
  \end{proof}
%CF added below
We note that, in $G'_1$, a 3-coloring $\varphi$ of $\{v_1,v_2,v_3,v_4, v_5\}$ such that at least two different colors appear on $v_1$, $v_2$, and $v_3$, and $\varphi(v_4)\neq\varphi(v_5)$ does not always extend to a $3$-coloring of $G'_1$. To allow this, we construct the graph $G_1$ depicted in Figure \ref{fig:gadget1}.  Since $G_1$ is obtained by combining graphs $G'_1$ and $G_0$, Lemmas~\ref{lem:g0} and \ref{lem:g1prime} have
the following consequence.

\begin{corollary}\label{cor:g1}
Let $G_1$ be the graph depicted in Figure \ref{fig:gadget1}.  A $3$-coloring $\varphi$ of $\{v_1,v_2,v_3,v_4, v_5\}$ extends to a $3$-coloring of $G_1$
if and only if either
\begin{itemize}
\item $\varphi(v_1)=\varphi(v_2)=\varphi(v_3)=\varphi(v_4)=\varphi(v_5)$, or
\item at least two different colors appear on $v_1$, $v_2$, and $v_3$, and $\varphi(v_4)\neq\varphi(v_5)$.
\end{itemize}
\end{corollary}

Finally, let us describe the properties of the gadget depicted in Figure \ref{fig:gadget2}.

 \begin{lemma}\label{lem:g2}
 In every $3$-coloring $\varphi$ of the graph $G_2$ depicted in Figure \ref{fig:gadget2}, $\varphi(u_2) = \varphi(u_3)$ and either 
 \begin{itemize}
 \item $\varphi(u_1) = \varphi(u_2)$ and $\varphi(u_3) \not= \varphi(u_4)$ or 
 \item $\varphi(u_1) \not= \varphi(u_2)$ and $\varphi(u_3) = \varphi(u_4)$. 
 \end{itemize}
 Moreover, any 3-coloring of $u_1$, \ldots, $u_4$ which assigns the same color to $u_2$, $u_3$, and exactly one of $u_1$ and $u_4$ extends to a 3-coloring of $G_2$.
 \end{lemma}
 
 \begin{proof}
 The easy verification is left to the reader. 
 \end{proof}

\section{The construction}

An edge $e$ of $G$ is \emph{critical} if $\chi(G-e)<\chi(G)$.  We are not actually going to explicitly
describe a 4-critical planar graph of the required
density.  Rather, we are going to show that the graph
that we construct has many critical edges.  This
is sufficient by the following observation.
 
 \begin{lemma}\label{lem:critical}
Let $G$ be a graph of chromatic number exactly $k$, and let $H$ be a $k$-critical subgraph of $G$.  Then $H$ contains all critical edges of $G$.
 \end{lemma}
 \begin{proof}
 Consider any edge $e\in E(G)\setminus E(H)$.  We have
 $\chi(G-e)\ge \chi(H)=k=\chi(G)$, and thus $e$ is not critical.
 \end{proof}
 
 \begin{figure}
\begin{center}\includegraphics[width=1\textwidth]{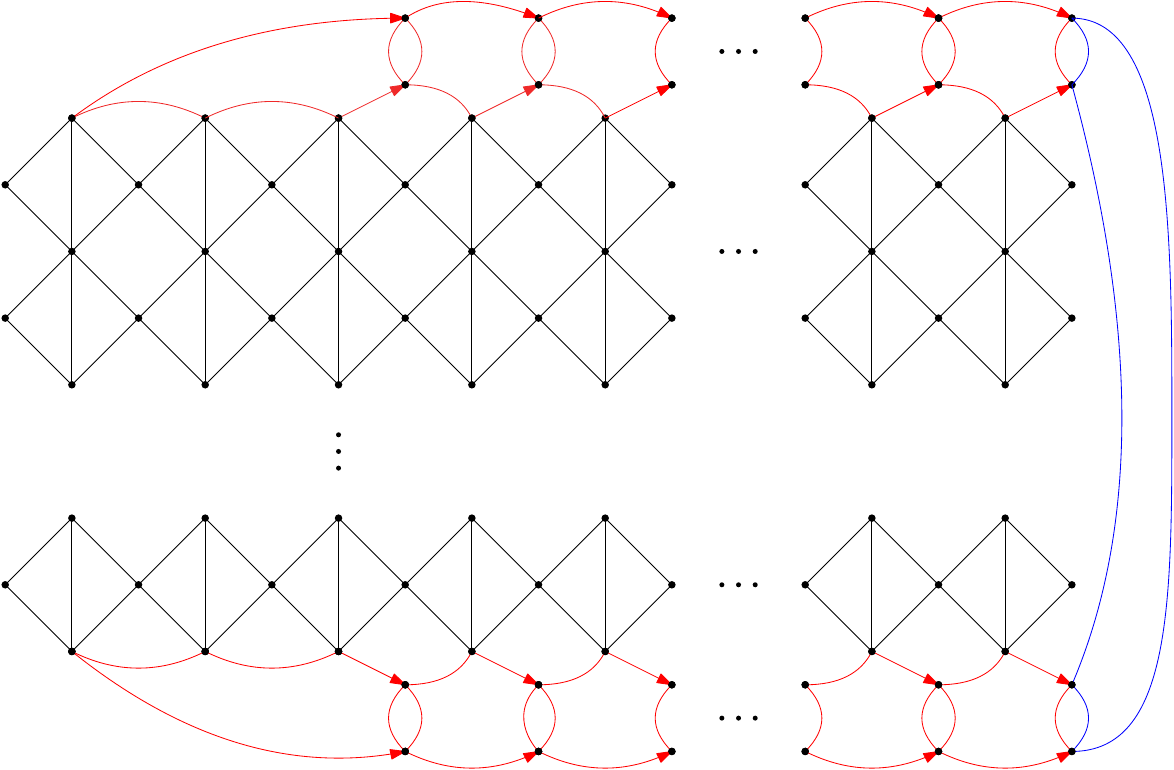}\end{center}
\caption{The graph $G_n$}
\label{fig:construction}
\end{figure}

For an integer $n \geq 4$, let $G_n$ be the graph that is obtained from 
\begin{itemize}
\item the graph $L_n$,  
\item $2(n- 2)$ copies $F_1, \dots, F_{n-2}, H_1, \dots, H_{n-2}$ of $G_1$ and
\item one copy $H$ of $G_2$
\end{itemize}
 as follows; see also Figure \ref{fig:construction} for an illustration. Let $V^\star(G_1) = \{v_1, v_2, v_3, v_4, v_5\}$ and $V^*(G_2) = \{u_1, u_2, u_3, u_4\}$. For $i = 4, \dots, n$, we set
\begin{gather*}
V^\star(F_{i-2}) = \{s_{i-2}, t_{i-2}, x_{1,i}, s_{i-1}, t_{i-1}\} \\
V^\star(H_{i-2}) = \{p_{i-2}, q_{i-2}, x_{n,i}, p_{i-1}, q_{i-1}\}
\end{gather*}
such that, \emph{in} $F_{i-2}$ and $H_{i-2}$, the vertices $s_{i-2}, t_{i-2}, s_{i-1}, t_{i-1}$ and the vertices $p_{i-2}, q_{i-2}, p_{i-1}, q_{i-1}$ stand for the vertices $v_1, v_2, v_4, v_5$, respectively,  in $G_1$, and where the vertices $x_{1, i}$ and $x_{n, i}$ are members of $V(L_n)$. We also set
\begin{gather*}
V^\star(F_1) = \{x_{1, 1}, x_{1, 2}, x_{1, 3}, s_2, t_2\} \\
V^\star(H_1) = \{x_{n, 1}, x_{n, 2}, x_{n, 3},p_2, q_2\} 
\end{gather*}
To complete the construction of $G_n$, we finally set \begin{gather*}
V^\star(H) = \{s_{n-1}, t_{n-1}, p_{n-1}, q_{n-1}\},
\end{gather*}
where $s_{n-1}, t_{n-1}, p_{n-1}, q_{n-1}$ stand for $u_1, u_2, u_3, u_4$, respectively, in $G_2$.  

We say that a $3$-coloring of a set of vertices is \emph{monochromatic} if all vertices in the set
receive the same color, and \emph{non-monochromatic} otherwise.
The key 3-coloring property of this construction is given in the following lemma.
\begin{lemma}\label{lem:bothmono}
Let $n\ge 4$ be an integer and let $X$ be a subset of $E(L_n)$.  A 3-coloring $\varphi$ of $L_n-X$
extends to a 3-coloring of $G_n-X$ if and only if $\varphi$ is monochromatic on exactly one of the
sets $\{x_{1,1},\ldots, x_{1,n}\}$ and $\{x_{n,1},\ldots, x_{n,n}\}$.
\end{lemma}
\begin{proof}
By Corollary~\ref{cor:g1}, $\varphi\restriction\{x_{1,1},\ldots, x_{1,n}\}$ can always be
extended to $F_1$, \ldots, $F_{n-2}$.
For $b,c\in [3]$, let $\varphi_{b,c}$ be such an extension with $\varphi_{b,c}(s_{n-1})=b$ and $\varphi_{b,c}(t_{n-1})=c$,
if one exists.  Corollary~\ref{cor:g1} implies that
\begin{itemize}
\item if $\{x_{1,1},\ldots, x_{1,n}\}$ is monochromatic and all its vertices receive color $a$,
then only $\varphi_{a,a}$ exists, and
\item if $\{x_{1,1},\ldots, x_{1,n}\}$ is non-monochromatic, then $\varphi_{b,c}$ exists exactly for
the colors $b,c\in [3]$ such that $b\neq c$.
\end{itemize}
Similarly, letting $\varphi'_{b,c}$ be an extension of $\varphi\restriction\{x_{n,1},\ldots, x_{n,n}\}$
to $H_1$, \ldots, $H_{n-2}$ with $\varphi'_{b,c}(p_{n-1})=b$ and $\varphi'_{b,c}(q_{n-1})=c$, if one exists,
and note that
\begin{itemize}
\item if $\{x_{n,1},\ldots, x_{n,n}\}$ is monochromatic and all its vertices receive color $a$,
then only $\varphi'_{a,a}$ exists, and
\item if $\{x_{n,1},\ldots, x_{n,n}\}$ is non-monochromatic, then $\varphi'_{b,c}$ exists exactly for
the colors $b,c\in [3]$ such that $b\neq c$.
\end{itemize}

Let $G'_n=(L_n-X)\cup F_1\cup\ldots\cup F_{n-2}\cup H_1\cup\ldots\cup H_{n-2}$, and let $\varphi'$ be
any extension of $\varphi$ to $G'_n$.
By the previous observations, if both $\{x_{1,1},\ldots, x_{1,n}\}$ and $\{x_{n,1},\ldots, x_{n,n}\}$ are monochromatic,
then $\varphi'(s_{n-1})=\varphi'(t_{n-1})$ and $\varphi'(p_{n-1})=\varphi'(q_{n-1})$ and by Lemma~\ref{lem:g2}, $\varphi'$ does not extend
to $H$.  If neither is monochromatic, then $\varphi'(s_{n-1})\neq\varphi'(t_{n-1})$ and $\varphi'(p_{n-1})\neq\varphi'(q_{n-1})$
and by Lemma~\ref{lem:g2}, $\varphi'$ again does not extend to $H$.

Hence, suppose that exactly one of $\{x_{1,1},\ldots, x_{1,n}\}$ and $\{x_{n,1},\ldots, x_{n,n}\}$ is monochromatic;
say all vertices of $\{x_{1,1},\ldots, x_{1,n}\}$ receive color $a$.  Let $b$ be any color in $[3]\setminus\{a\}$.
By Lemma~\ref{lem:g2}, $\varphi\cup \varphi_{a,a}\cup \varphi'_{a,b}$ extends to $H$, giving a $3$-coloring of $G_n-X$ extending $\varphi$.
\end{proof}

From this, it is easy to see that $G_n$ is not $3$-colorable.

\begin{lemma}\label{claim:0}
For every $n\ge 4$, the graph $G_n$ has chromatic number
four.
\end{lemma}

\begin{proof}
Since $G$ is planar, it is 4-colorable.
If $\varphi$ were a $3$-coloring of $G_n$, then
by Lemma~\ref{lem:bothmono} with $X=\emptyset$, exactly one of the sets $\{x_{1,1},\ldots, x_{1,n}\}$
and $\{x_{n,1},\ldots, x_{n,n}\}$ would be monochromatic.  However, this contradicts Lemma~\ref{lem:L}.
\end{proof} 

Next, we need to show that most of the edges of the copy of $L_n$ in $G_n$ are
critical.  Up to symmetry, $L_n$ has two types of edges (vertical and
diagonal), and we discuss each type separately, starting with the vertical
ones.

For $s, t \in [n]$, a coloring of $\{x_{t,1}, \dots, x_{t,n}\}$ is a \emph{$(t;
a; s;  b)$-motif} if it colors vertices $x_{t,1}, \dots, x_{t,s-1}, x_{t,s+1},
\dots, x_{t,n}$ with $a$ and  vertex $x_{t,s}$ with $b$.
A coloring of all vertices of $\{x_{t,1}, \dots, x_{t,n}\}$ by the same color $a$
is called a \emph{$(t;a)$-smear}. 

\begin{lemma}\label{claim:1}
For every $n\ge 4$, $i\in[n-1]$, and $j\in[n]$, the edge $e=x_{i, j}x_{i+1, j}$ of $G_n$ is critical.
\end{lemma}  

\begin{proof}
We need to describe a $3$-coloring $\varphi$ of $G_n-e$.

On $L_n-e$, the 3-coloring $\varphi$ is defined as follows.
For $k = n, \dots, i+1$ in order, alternate between a $(k; 1; j; 2)$-motif and
a $(k; 2; j; 1)$-motif.  Let $a\in\{1,2\}$ be the color assigned to $x_{i+1, j}$.
For $k = i, \dots, 1$ in order, alternate between a $(k; a)$-smear and a $(k; 3-a)$-smear
(using the $(i;a)$-smear is possible, since the edge $e=x_{i, j}x_{i+1, j}$ is missing).
The vertices of all stripes are given color $3$.

Thus, the set $\{x_{1,1},\ldots,x_{1,n}\}$ is monochromatic and the set $\{x_{n,1},\ldots,x_{n,n}\}$
is non-monochromatic.  By Lemma~\ref{lem:bothmono} with $X=\{e\}$, $\varphi$ extends to a $3$-coloring of $G_n-e$.
\end{proof}

Next, we consider the \emph{diagonal} edges, i.e., the edges joining a vertex $x_{i,j}$ with a vertex $y_{i',j'}$
for some indices $i,j\in[n]$, $i'\in[n-1]$, and $j'\in[n+1]$.
For $s, t \in [n]$, a coloring of $\{x_{t,1}, \dots, x_{t,n}\}$ is a \emph{$(t; a; s;  b)$-pattern} 
if it colors vertices $x_{t,1}, \dots, x_{t,s}$ with $a$ and vertices $x_{t,s+1}, \dots, x_{t,n}$ with $b$.

\begin{lemma}\label{claim:2}
For every $n\ge 4$, every diagonal edge $e$ of $G_n$ not incident with a vertex of degree two is critical.
\end{lemma}
\begin{proof}
By symmetry, we can assume that $e=x_{i, j}y_{i-1, j+1}$, where $i\in [n]\setminus\{1\}$ and $j\in[n-1]$.
We need to describe a $3$-coloring $\varphi$ of $G_n-e$.

On $L_n-e$, the 3-coloring $c$ is defined as follows.
For $k = n, \dots, i$ in order, alternate between a $(k; 1; j; 2)$-pattern and a
$(k; 2; j; 1)$-pattern.   
For $k = i-1, \dots, 1$ in order, alternate between a $(k; 3)$-smear and a $(k; 2)$-smear.
The vertices of stripes $n-1$, \ldots, $i$ get color $3$.
Let $a\in \{1,2\}$ be the color assigned to $x_{i, j}$.  The vertices $\{y_{i-1,1},\ldots,y_{i-1,j}\}$
get color $3-a$ and the vertices $\{y_{i-1,j+1},\ldots,y_{i-1,n+1}\}$ get color $a$ (this is possible,
since the edge $e=x_{i, j}y_{i-1, j+1}$ is missing). The vertices of stripes $i-2$, \ldots, $1$
get color $1$.

Thus, the set $\{x_{1,1},\ldots,x_{1,n}\}$ is monochromatic and the set $\{x_{n,1},\ldots,x_{n,n}\}$
is non-monochromatic.  By Lemma~\ref{lem:bothmono} with $X=\{e\}$, $\varphi$ extends to a $3$-coloring of $G_n-e$.
\end{proof}

We are now ready to prove the existence of 4-critical graphs of the required density.

\begin{proof}[Proof of Theorem~\ref{thm:main}]
Let $k_1$ be the number of vertices of $G_1$ and $k_2$ the number of vertices of $G_2$.
Note that for every $n\ge 4$, we have
\begin{align*}
|V(G_n)|&=|V(L_n)|+2(n-2)(k_1-3)+k_2-4\\&=2n^2+2(n-2)(k_1-3)+k_2-5\\
&=2n^2+an+b,
\end{align*}
where $a=2(k_1-3)$ and $b=k_2-4k_1+7$.
Let $Q_n$ be a 4-critical subgraph of $G_n$; by Lemma~\ref{lem:critical}, $Q_n$ contains
all critical edges of $G_n$, and by Lemmas~\ref{claim:1} and \ref{claim:2}, we conclude
that
$$|E(Q_n)|\ge (n-1)n+4(n-1)^2=5n^2-9n+4.$$
Therefore,
$$\frac{|E(Q_n)|}{|V(Q_n)|}\ge \frac{|E(Q_n)|}{|V(G_n)|}\ge \frac{5n^2-9n+4}{2n^2+an+b}\ge 2.5-O(1/n).$$
It follows that
$$L=\limsup_{\text{$G$ 4-critical}} \frac{E(G)}{V(G)}\ge \limsup_{n\ge 4} \frac{E(Q_n)}{V(Q_n)}\ge 2.5.$$
The bound $L\le 2.5$ has been proven by Koester~\cite{koester2}.
 \end{proof}

\section*{Acknowledgements}
Zden\v{e}k Dvo\v{r}\'ak was supported by project 22-17398S (Flows and cycles in graphs on surfaces) of Czech Science Foundation. 
Carl Feghali was supported by the French National Research Agency under research grant ANR DIGRAPHS ANR-19-CE48-0013-01.

\bibliography{bibliography}{}

\begin{thebibliography}{10}

\bibitem{abbott}
H.~Abbott and B.~Zhou.
\newblock The edge density of 4-critical planar graphs.
\newblock {\em Combinatorica}, 11:185--189, 1991.

\bibitem{dirac}
G.~Dirac.
\newblock A theorem of {R. L.} {B}rooks and a conjecture of{ H. Hadwiger}.
\newblock {\em Proceedings of the London Mathematical Society}, 3(1):161--195,
  1957.

\bibitem{gallai1963c}
T.~Gallai.
\newblock Critical graphs.
\newblock In {\em Theory of Graphs and its Applications, Proceedings of the
  Symposium held in Smolenice in June}, pages 43--45, 1963.

\bibitem{gallai1963}
T.~Gallai.
\newblock Kritische graphen ii.
\newblock {\em Publ. Math. Inst. Hungar. Acad. Sci.}, 8:373--395, 1963.

\bibitem{grunbaum}
B.~Gr{\"u}nbaum.
\newblock The edge-density of 4-critical planar graphs.
\newblock {\em Combinatorica}, 8:137--139, 1988.

\bibitem{koester1}
G.~Koester.
\newblock Note to a problem of {T. Gallai and G. A. Dirac}.
\newblock {\em Combinatorica}, 5(3):227--228, 1985.

\bibitem{koester2}
G.~Koester.
\newblock On 4-critical planar graphs with high edge density.
\newblock {\em Discret. Math.}, 98(2):147--151, 1991.

\bibitem{kostochka}
A.~Kostochka and M.~Yancey.
\newblock Ore's conjecture on color-critical graphs is almost true.
\newblock {\em Journal of Combinatorial Theory, Series B}, 109:73--101, 2014.

\bibitem{li20223}
J.~Li, Y.~Ma, Y.~Shi, W.~Wang, and Y.~Wu.
\newblock On 3-flow-critical graphs.
\newblock {\em European Journal of Combinatorics}, 100:103451, 2022.

\bibitem{ore}
O.~Ore.
\newblock {\em The four-color problem}.
\newblock Academic Press, 2011.

\bibitem{stiebitz}
M.~Stiebitz.
\newblock Subgraphs of colour-critical graphs.
\newblock {\em Combinatorica}, 7:303--312, 1987.

\bibitem{toft}
B.~Toft.
\newblock On the maximal number of edges of critical k-chromatic graphs.
\newblock {\em Studia Sci. Math. Hung}, 5:461--470, 1970.

\bibitem{yao}
T.~Yao and G.~Zhou.
\newblock Constructing a family of 4-critical planar graphs with high edge
  density.
\newblock {\em Journal of Graph Theory}, 86(2):244--249, 2017.

\end{thebibliography}
\bibliographystyle{abbrv}
 
\end{document}